\documentclass{article}

\title{Iterative reconstruction of signals on graph}
\author{Emanuele Brugnoli$^1$ \and Elena Toscano$^2$ \and Calogero Vetro$^2$}
\date{%
    $^1$Institute for Complex Systems (ISC), Rome 00185 Italy\\%
    $^2$University of Palermo, Palermo 90123 Italy\\[2ex]%
    \today
}

\usepackage[a4paper, total={6in, 10in}]{geometry}

\usepackage{graphicx,url}
\usepackage{amsmath,amssymb,amsthm}
\usepackage{algorithm}
\usepackage{algorithmic}
\makeatletter
\newcommand\fs@norules{\def\@fs@cfont{\bfseries}\let\@fs@capt\floatc@ruled
  \def\@fs@pre{}%
  \def\@fs@post{}%
  \def\@fs@mid{\kern3pt}%
  \let\@fs@iftopcapt\iftrue}
\makeatother
\floatstyle{norules}
\restylefloat{algorithm}

\newtheorem{theorem}{Theorem}[section]
\newtheorem{proposition}{Proposition}[section]

\theoremstyle{definition}
\newtheorem{definition}{Definition}[section]

\begin{document}
\maketitle

\begin{abstract}
We propose an iterative algorithm to interpolate graph signals from only a partial set of samples. Our method is derived from the well known Papoulis-Gerchberg algorithm by considering the optimal value of a constant involved in the iteration step. Compared with existing graph signal reconstruction algorithms, the proposed method achieves similar or better performance both in terms of convergence rate and computational efficiency.
\end{abstract}

\section{Introduction}
\label{se:intro}
\subsection{Graph signal processing}
The ongoing wide availability of information and communication spreading on social, energy, transportation and sensor networks, among others, requires efficient approaches to process signals defined on irregular domains. To this aim, the emerging field of signal processing on graphs extends classical discrete signal processing to signals with graphs as underlying structure \cite{sand,sand2,CC}.

More formally, let $\mathcal{G=(V,E)}$ be an undirected graph with finite vertex set $\mathcal{V}=\{x_i\}_{i=1}^N$ and edge set $\mathcal{E}$.
For any $x\in\mathcal{V}$, we define the \textit{degree} of $x$ to be the number of edges connected to $x$ and denote it with $d_x$.\\
If one complex number is associated with each vertex, all these numbers are collectively referred as a \textit{graph signal}. Thus, a graph signal is a mapping $f\colon\mathcal{V}\to\mathbb{C},\,x_n\mapsto f(x_n)$ ($f(n)$, for short), and the complex $N$-dimensional space $\mathbb{C}^N$ represents the space of all $N$-dimensional graph signals.
\subsection{Related works and contributions}
The task of sampling and recovery smooth signals from partial observations is one of the most investigated topics in signal processing on graphs. Some works focus on the teorethical conditions for the exact reconstruction of bandlimited signals \cite{pesen,nara}, other works focus on techniques for an efficient sampling set selection \cite{anis,chen}, and several methods have been proposed to reconstruct bandlimited graph signals from sampled data \cite{chen2,marq,BB,SMLR,TBDL,DD,wang2,zhu}. A method derived from projection onto convex sets \cite{youl} and called \textit{Iterative Least Square Reconstruction} (ILSR) is presented in \cite{BB}. A \textit{frame-based} representation \cite{chri} of ILSR is given in \cite{DD} together with the definition of \textit{local sets}, that is a partition of the vertex set into disjoint neighbours of the sample vertices. Based on these concepts, the authors of \cite{DD} also proposed two algorithms, namely \textit{Iterative Weighting Reconstruction} (IWR) and \textit{Iterative Propagating Reconstruction} (IPR), that perfectly recover $\omega$-bandlimited signals for any $\omega$ less than or equal to a certain measure of the local sets. IWR and IPR are derived from the Adaptive Weights Method \cite{feic} and the Voronoi Method \cite{groc3}, respectively, and are shown to converge faster than ILSR. All the aforementioned algorithms are derived from the classical Papoulis-Gerchberg iterative scheme \cite{gerc,papou} with a unitary relaxation parameter involved in the iteration step. 

In this work we exploit the same scheme to interpolate bandlimited graph signals from only a partial set of samples, but we consider the optimal value of the relaxation parameter. Compared with the other methods, the proposed algorithm achieves similar or better performance both in terms of convergence rate and computational efficiency.

The paper is organized as follows. Section \ref{se:pre} covers some preliminaries of graph signal processing, matrix operators and general iterative schemes. In Section \ref{subse:papo}, the classical Papoulis-Gerchberg iteration is introduced together with a result that guarantees its convergence. In Section \ref{se:reconstruction}, an iterative reconstruction method O-PGIR is proposed and its convergence rate is analyzed. Section \ref{subse:exper} presents some numerical experiments. 

\section{Preliminaries}\label{se:pre}
\subsection{Graph Laplacian and bandlimited graph signals}
Let $D$ denote the degree matrix of a graph $\mathcal{G}$ which is the diagonal $N\times N$ matrix given by $D=\text{diag}(d_x)$, and let $A$ denote the adjacency matrix of $\mathcal{G}$, also $N\times N$, where
\[
A(i,j)=
\begin{cases}
1,\quad\text{if }(x_i,x_j)\in\mathcal{E},\\
0,\quad\text{otherwise}. 
\end{cases}
\]
Then the \textit{Laplacian} of $\mathcal{G}$ can be written as $L=D-A$. Matrix $L$ is called Laplacian to distinguish it from the \textit{normalized Laplacian} $\mathcal{L}=D^{-1/2}LD^{-1/2}$. We consider exclusively the normalized Laplacian matrix $\mathcal{L}$ because it is shown to produce superior classification results \cite{zhou} and, under the assumption $\mathcal{G}$ undirected without self loops, it is a symmetric positive semi-definite matrix. Therefore, it has nonnegative eigenvalues $\{\lambda_k\}_{k=1}^{N}$ always in the range $[0,2]$ with associated orthonormal real-valued eigenvectors $\{\varphi_k\}_{k=1}^{N}$. Let us also denote by $\varphi^{\ast}$ the conjugate transpose of a vector $\varphi$, by $\Phi$ the $N\times N$ orthogonal matrix whose $k$-th column is $\varphi_k$ and by $\sigma(\mathcal{L})$ the spectrum of $\mathcal{L}$.

The \textit{graph Fourier transform} (GFT) $\hat{f}$ of a signal $f\colon\mathcal{V}\to\mathbb{C}$ on $\mathcal{G}$ is only defined on values of $\sigma(\mathcal{L})$  as the expansion of $f$ in terms of $\{\varphi_k\}$, that is $\hat{f}=\Phi^{\ast}f$.
Similar with classical Fourier analysis, eigenvalues $\{\lambda_k\}_{k=1}^{N}$ represent frequencies of the graph, and $\hat{f}(\lambda_l)$ represents the frequency component corresponding to $\lambda_l$. Low-frequency (High-frequency) components are associated with smaller (larger) eigenvalues. The \textit{inverse graph Fourier transform} is then given by $f=\Phi\hat{f}$. The matrices $\Phi^{\ast}$ and $\Phi$ are called \textit{graph Fourier matrix} and \textit{inverse graph Fourier matrix}, respectively. The inverse graph Fourier transform guarantees a perfect reconstruction of $f$ from the knowledge of $\hat{f}(\lambda_l)$ for all the eigenvalues $\lambda_l$ of $\mathcal{L}$. However if we only reconstruct $f$ from values of $\hat{f}(\lambda_l)$ with low magnitudes, we obtain an approximation to $f$.

If supp$(\hat{f})\subseteq[0,\omega]$, $f$ is called \textit{$\omega$-bandlimited} and the subspace $PW_{\omega}(\mathcal{G})$ of $\omega$-bandlimited signals on $\mathcal{G}$ is called \textit{Paley-Wiener space}. Bandlimited signals are smooth, and the smoothness increases as the bandwidth decreases. Suppose that for a graph signal $f\in PW_{\omega}(\mathcal{G})$, only $\{f(u)\}_{u\in\mathcal{S}}$ on the sampling set $\mathcal{S}\subseteq\mathcal{V}$ are known. In \cite{pesen}, Pesenson showed that $f$ can be uniquely reconstructed from its entries $\{f(u)\}_{u\in\mathcal{S}}$ under certain conditions. To make a practical use of Pesenson's result, the authors of \cite{nara} presented the following result to compute the maximum $\omega$.
\begin{proposition}\label{pr:smallest_eigen}
Given a graph $\mathcal{G=(V,E)}$ with normalized Laplacian matrix $\mathcal{L}$, known sampling set $\mathcal{S}$ and unknown set $\mathcal{S}^{c}=\mathcal{V\setminus S}$, let $(\mathcal{L}^2)_{\mathcal{S}^{c}}$ be the submatrix of $\mathcal{L}^2$ containing only the rows and columns corresponding to $\mathcal{S}^{c}$. Then any signal $f\in PW_{\omega}(\mathcal{G})$ with $\omega\leq\sigma_{\min}$, where $\sigma_{\min}^2$ is the smallest eigenvalue of $(\mathcal{L}^2)_{\mathcal{S}^{c}}$, can be uniquely recovered from its samples on $\mathcal{S}$.
\end{proposition}
Discrete signals are processed by \textit{filters}, i. e., systems that take a signal as input and produce another signal as output.
We represent filtering on a graph using matrix-vector multiplication. In particular, we will consider the linear operations called sampling and bandlimiting, respectively, and defined on $\mathbb{C}^N$ as follows. The \textit{sampling} filter maps a signal into another by setting to zero a certain subset of its samples. This corresponds to multiplication by a binary diagonal matrix $S$ called \textit{sampling matrix}. The \textit{density} of a sampling set is $s/N$, $s$ being the number of nonzero entries in the sampling set.\\
The \textit{bandlimiting} filter onto $PW_{\omega}(\mathcal{G})$ is characterized by an idempotent matrix $P$ of the form $P=\Phi S\Phi^{\ast}$, where $S$ is a sampling matrix other than the identity $I$.
Signals that satisfy $f=Pf$ are low-pass signals and $P$ is said \textit{low-pass filter matrix}.

\subsection{Matrix properties and concepts}
Hereafter, we will consider the space of all $N$-dimensional graph signals $\mathbb{C}^N$ endowed with the metric induced by the usual norm $\left\lVert\cdot\right\rVert$.
\begin{definition}
Let $\rho(A)$ be the \textit{spectral radius} of an arbitrary matrix $A$, that is, the greatest of its eigenvalues in absolute value. The \textit{spectral norm} of $A$ is given by
\begin{equation}\label{eq:spe_nor}
\left\lVert A\right\rVert=\displaystyle\sup_{\left\lVert v\right\rVert=1}\left\lVert Av\right\rVert=\sqrt{\rho(A^{\ast}A)}.
\end{equation}
\end{definition}
\begin{definition}\label{de:nonexp}
A matrix $A$ defined on $\mathbb{C}^N$ is \textit{nonexpansive} if $\left\lVert Au-Av\right\rVert\leq \left\lVert u-v\right\rVert$ for all $u,v\in\mathbb{C}^N$, and \textit{strictly nonexpansive} if equality holds only for $u=v$.
\end{definition}
Nonexpansiveness of $A$ means $\left\lVert Av\right\rVert\leq\left\lVert v\right\rVert$ for any $v\in\mathbb{C}^N$. Thus, the spectral norm of a nonexpansive matrix $A$ cannot exceed unity. It follows from $\rho(A)\leq\left\lVert A\right\rVert$ that all the eigenvalues of $A$ do not exceed unity (in absolute value). Therefore, a strictly nonexpansive matrix $A$ is convergent to zero \cite{lim}.
The point $v\in\mathbb{C}^N$ is a \textit{fixed point} of $A$ if $Av=v$.

Consider now the general linear stationary iterative algorithm of first order 
\begin{equation}
v^{(i+1)}=Av^{(i)}+b,
\end{equation}
where $A$ is a square matrix defined on $\mathbb{C}^N$ and $b\in\mathbb{C}^N$. Assuming $A$ convergent to zero, that is $\rho(A)<1$, the error vector $e^{(k)}=v^{(k+1)}-v$ at iteration $k$ is given by $e^{(k)}=Ae^{(k-1)}$ and it is inductively related to the initial error $e^{(0)}$ by $e^{(k)}=A^ke^{(0)}$. Thus
\begin{equation}
\lVert e^{(k)}\rVert\leq\lVert A^k\rVert\lVert e^{(0)}\rVert.
\end{equation}
Since $e^{(0)}$ is unknown in practical problems, $\lVert A^k\rVert$ must be studied for comparing different iterative algorithms. The simplest measure of rapidity of convergence of $A$ is the \textit{rate of convergence} $R_{\infty}(A)$ defined as follows \cite{varg}.
\begin{definition}
Let $A$ be a $N\times N$ complex matrix. If $\lVert A^k\rVert<1$ for some integer $k>0$, then $R(A^k)=\frac{-\ln\lVert A^k\rVert}{k}$ is the average rate of convergence for $k$ iterations of $A$.
\end{definition}
\begin{theorem}\label{th:conv_rate}
If $A$ is a convergent $N\times N$ complex matrix, then $R(A^k)$ satisfies
\begin{equation}
R_{\infty}(A)=\lim_{k\to\infty}R(A^k)=-\ln\rho(A)
\end{equation}
and $R_{\infty}(A)$ represents the rate of convergence for $A$.
\end{theorem}
Therefore, the lower is the spectral radius of $A$, the faster the convergence of $A$.

\section{Papoulis-Gerchberg Iterative Reconstruction}\label{subse:papo}

The \textit{Papoulis-Gerchberg Iterative Reconstruction} (PGIR) \cite{gerc,papou} has been extensively used to solve the missing data problem in bandlimited signals \cite{AA}. Each iteration of PGIR consists of two steps: the low-pass filtering step
\begin{equation}\label{eq:P}
p^{(k)}=Pf^{(k)}
\end{equation}
which imposes the frequency domain constraints about the data, followed by the resampling step which restores the known data
\begin{equation}\label{eq:T}
f^{(k+1)}=Tp^{(k)}.
\end{equation}
The resampling operator is defined by $T(\cdot)=\mu f^{(1)}+(I-\mu S)(\cdot)$, where $S$ is a sampling matrix and $\mu$ is a fixed constant called \textit{relaxation parameter}. The original signal satisfies $f=Pf$ and the observed signal is $f^{(1)}=Sf$. Combining (\ref{eq:P}) and (\ref{eq:T}) we have
\begin{equation}\label{eq:papoulis}
f^{(k+1)}=\mu Sf+(I-\mu S)Pf^{(k)}=T^{\prime}f^{(k)}
\end{equation}
where the operator
\begin{equation}\label{eq:op_one}
T^{\prime}(\cdot)=\mu Sf+(I-\mu S)P(\cdot)
\end{equation}
is nonexpansive for $0\leq\mu\leq 2$. Indeed, by letting
\begin{equation}\label{eq:a_mu}
A_{\mu}=(I-\mu S)P
\end{equation}
it follows from $(\ref{eq:op_one})$ that $T^{\prime}$ is nonexpansive if and only if $A_{\mu}$ is. Since $P$ is obviously nonexpansive, the nonexpansiveness of $A_{\mu}$ depends only on $I-\mu S$, which in turn is nonexpansive if and only if $0\leq\mu\leq 2$.

The following result \cite{ferre} guarantees the strict nonexpansiveness of $T^{\prime}$ and then the convergence of (\ref{eq:papoulis}) if suitable conditions are imposed upon $S, P$ and $\mu$.
\begin{theorem}\label{te:only_one_sol}
Let $S$ be a sampling matrix with density $d$, and let $P$ be a $w$-bandlimiting matrix. If
\begin{equation}\label{eq:muu}
0<\mu<2\quad\text{and}\quad d\geq\omega,
\end{equation}
then $T^{\prime}$ is strictly nonexpansive and the iteration (\ref{eq:papoulis}) converges to the unique fixed point of $T^{\prime}$ for arbitrary $f^{(1)}$.
\end{theorem}

\section{The proposed algorithm}\label{se:reconstruction}

Theorem \ref{te:only_one_sol} provides an iterative method for the reconstruction of bandlimited graph signals from an appropriate sampling set.
\begin{proposition}
Given a graph $\mathcal{G=(V,E)}$, let $S$ be the sampling matrix associated with a sampling set $\mathcal{S}$ of density $d$, and let $P$ be the low-pass filter matrix onto $PW_{\omega}(\mathcal{G})$ with $\omega\leq\sigma_{\min}$ as in Proposition \ref{pr:smallest_eigen}. Then, under the assumptions (\ref{eq:muu}), any $f\in PW_{\omega}(\mathcal{G})$ can be uniquely recovered from its entries $\{f(u)\}_{u\in\mathcal{S}}$ through the iteration (\ref{eq:papoulis}). Moreover, by recalling (\ref{eq:a_mu}), the best convergence rate of (\ref{eq:papoulis}) is obtained when
\begin{equation}\label{eq:err_bou}
\mu=\mu_{\text{opt}}=\frac{2}{2-\rho(A_1)}.
\end{equation}
\end{proposition}
\begin{proof}
The first statement follows from Proposition \ref{pr:smallest_eigen} and Theorem \ref{te:only_one_sol}.
From Theorem \ref{th:conv_rate}, it follows that the best convergence rate of (\ref{eq:papoulis}) is obtained for the value of $\mu$ which minimizes $\rho(A_{\mu})$. To this aim, consider $T_{\mu}=P(I-\mu S)P$ as in \cite{marv2}.
By left multiplying $A_{\mu}v=\lambda v$ by $P$, we have $T_{\mu}Pv=\lambda Pv$, and thus $\rho(A_{\mu})\leq\rho(T_{\mu})$.
Moreover, by left multiplying $T_{\mu}v=\lambda v$ by $A_{\mu}$, we obtain $A_{\mu}^2v=\lambda A_{\mu}v$. This shows that $\rho(T_{\mu})\leq\rho(A_{\mu})$. Combining the two results, it follows that $\rho(A_{\mu})=\rho(T_{\mu})$. Now, by left multiplying $T_1v=\lambda v$ by $P$, we have $Pv=v$ and also $P(I-S)Pv=v-PSPv=\lambda v$. Therefore, if $v$ is an eigenvector of $T_1$ referring to $\lambda$ then v is an eigenvector of $PSP$ referring to $1-\lambda$. This implies $P(I-\mu S)Pv=v-\mu PSPv=v-\mu(1-\lambda)v$, thus $v$ is also an eigenvector of $T_{\mu}$ pertaining to $1-\mu(1-\lambda)$. This shows that increasing $\mu$ towards $2$ reduces the spectral radius leading to better convergence rates, then we can assume $\mu>1$. Under the conditions of Theorem \ref{te:only_one_sol}, it is $\rho(A_1)<1$ and the smallest eigenvalue of $A_1$ is zero. Then the optimal value of $\mu$ is obtained by minimizing
\[
\rho(A_{\mu})=\max_{1<\mu<2}\{\mu-1,1-\mu(1-\rho(A_1))\}.
\]
Thus, (\ref{eq:err_bou}) follows by solving $\mu_{\text{opt}}-1=1-\mu_{\text{opt}}(1-\rho(A_1))$.
\end{proof}

The algorithm proposed is based on the iteration ($\ref{eq:papoulis}$) with $\mu=\mu_{\text{opt}}$. We call it \textit{Optimal Papoulis-Gerchberg Iterative Reconstruction} (O-PGIR). A pseudocode is displayed below, where $\sigma_{\min}$ is the maximal cutoff frequency defined as in Proposition \ref{pr:smallest_eigen}. 
\begin{algorithm}[H]
\small
 \caption{Optimal Papoulis-Gerchberg Iterative Reconstruction (O-PGIR)}
 \begin{algorithmic}[1]\label{al:O-PGIR}
 \renewcommand{\algorithmicrequire}{\textbf{Input:}}
 \renewcommand{\algorithmicensure}{\textbf{Output:}}
 \REQUIRE Graph $\mathcal{G}$, sampling set $\mathcal{S}$ with density $d$, sampled data $\{f(u)\}_{u\in\mathcal{S}}$, cutoff frequency $\omega\leq\min\{\sigma_{\min},d\}$;
 \ENSURE  Interpolated signal $f^{(k)}$;
 \\ \textit{Initialization}:
  \STATE $A_{1}=(I-S)P$;
  \STATE $\mu_{\text{opt}}=\frac{2}{2-\rho(A_1)}$;
  \STATE $A_{\mu_{\text{opt}}}=(I-\mu_{\text{opt}}S)P$;
  \STATE $f^{(1)}=\mu_{\text{opt}}Sf$;
 \\ \textit{Loop:}
  \STATE $f^{(k+1)}=f^{(1)}+A_{\mu_{\text{opt}}}f^{(k)}$;
 \\ \textit{Until:} The stop condition is satisfied.
 \end{algorithmic}
 \end{algorithm}

\section{Experimental results}\label{subse:exper}

\subsection{Performance analysis}
An Erd\"{o}s–R\'{e}nyi random graph $\mathcal{G}$ \cite{erdos} with $3000$ vertices and $12000$ edges and the Minnesota road graph $\mathcal{M}$ \cite{glei} which has $2640$ vertices and $6604$ edges, are chosen as  underlying sample structures to compare the performance of the proposed O-PGIR with the existing ILSR, IWR and IPR. For the sampling set, $35\%$ of vertices are selected uniformly at random among all the vertices. The bandlimited signal $f$ is generated by first considering a random signal and then removing its high-frequency components. The convergence curves of ILSR, IWR, and IPR and the proposed O-PGIR are illustrated in Fig. \ref{fi:it_number} and Fig. \ref{fi:exe_time} with respect to number of iterations and execution time, respectively, both on graph $\mathcal{G}$ (left panels) and graph $\mathcal{M}$ (right panels).
\begin{figure}[!htbp]
\centerline{\includegraphics[width=\columnwidth]{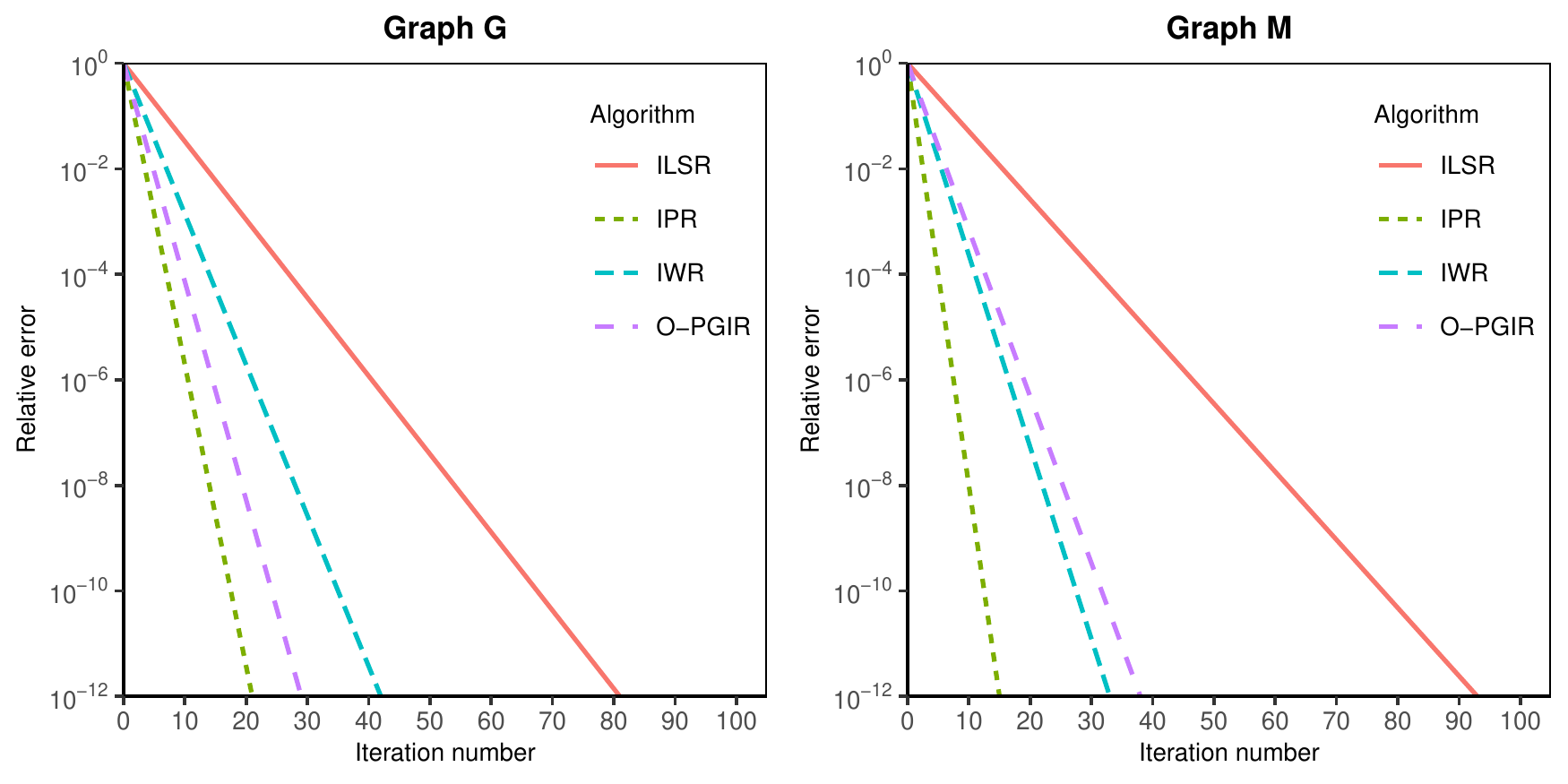}}
\caption{Left: graph $\mathcal{G}$. Right: graph $\mathcal{M}$. Comparison of the algorithms with respect to iteration number.}\label{fi:it_number}
\end{figure}

\begin{figure}[!htbp]
\centerline{\includegraphics[width=\columnwidth]{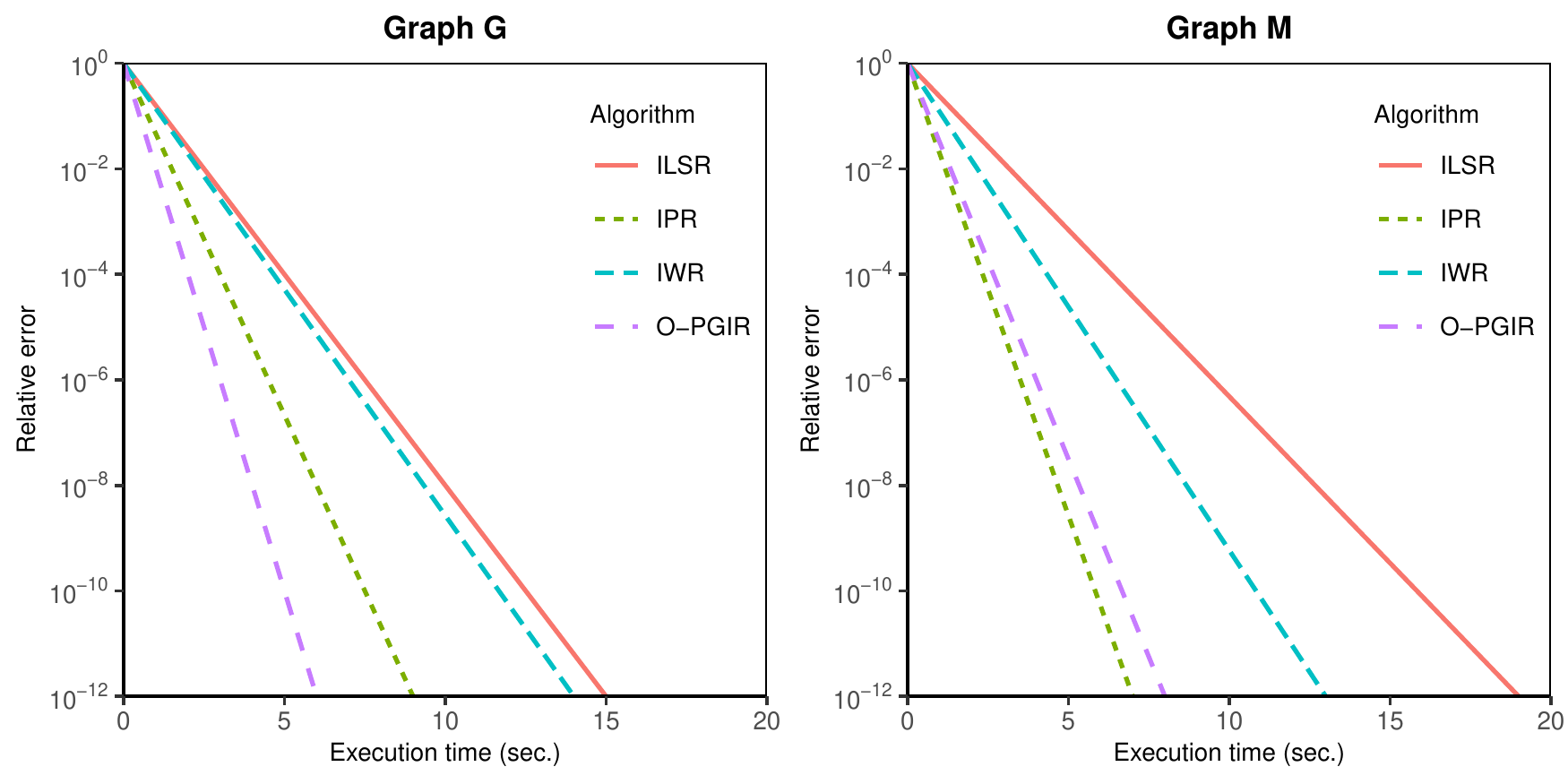}}
\caption{Left: graph $\mathcal{G}$. Right: graph $\mathcal{M}$. Comparison of the algorithms with respect to execution time.}\label{fi:exe_time}
\end{figure}
The simulation has been repeated for $100$ randomly generated signals and plots represent the mean values of all the simulation results. The code was implemented in MATLAB. Plots clearly show the importance of the relaxation parameter $\mu$ in the performance of convergence of (\ref{eq:papoulis}). Indeed, O-PGIR and ILSR, which basically differ only for the value of the relaxation parameter in (\ref{eq:papoulis}) ($\mu=\mu_{\text{opt}}$ and $\mu=1$, respectively), exhibit a very different convergence speed. Moreover, O-PGIR performs well even when compared with IWR and the reference algorithm IPR, specially with respect to the execution time. This is mainly due to the fact that, at each iteration, IWR and IPR require a propagating step that increases their computational complexity, whereas in O-PGIR (and ILSR) the most expensive, time-consuming operations are perfomed once outside the loop.

\subsection{Robustness against noise}
Here we want to investigate the robustness of O-PGIR in presence of noise. To this aim, we corrupt the observed signal with uncorrelated additive Gaussian noise and we compare the error perfomance of O-PGIR with the other algorithms. Fig. \ref{fi:noise} shows that the steady-state error decreases as the SNR increases, suggesting that all the methods have almost the same behavior against observation noise. Plots also shows that the interference of noise significantly limits the performance of all the methods. Indeed, the proposed method and the compared algorithms are all derived from PGIR which in turn usually runs under a noise-free condition. However, some recent studies on PGIR and its generalizations show how the use of some error tolerant techniques leads to reasonable results in the presence of noise too (see \cite{cho}).

\begin{figure}[!htbp]
\centerline{\includegraphics[width=\columnwidth]{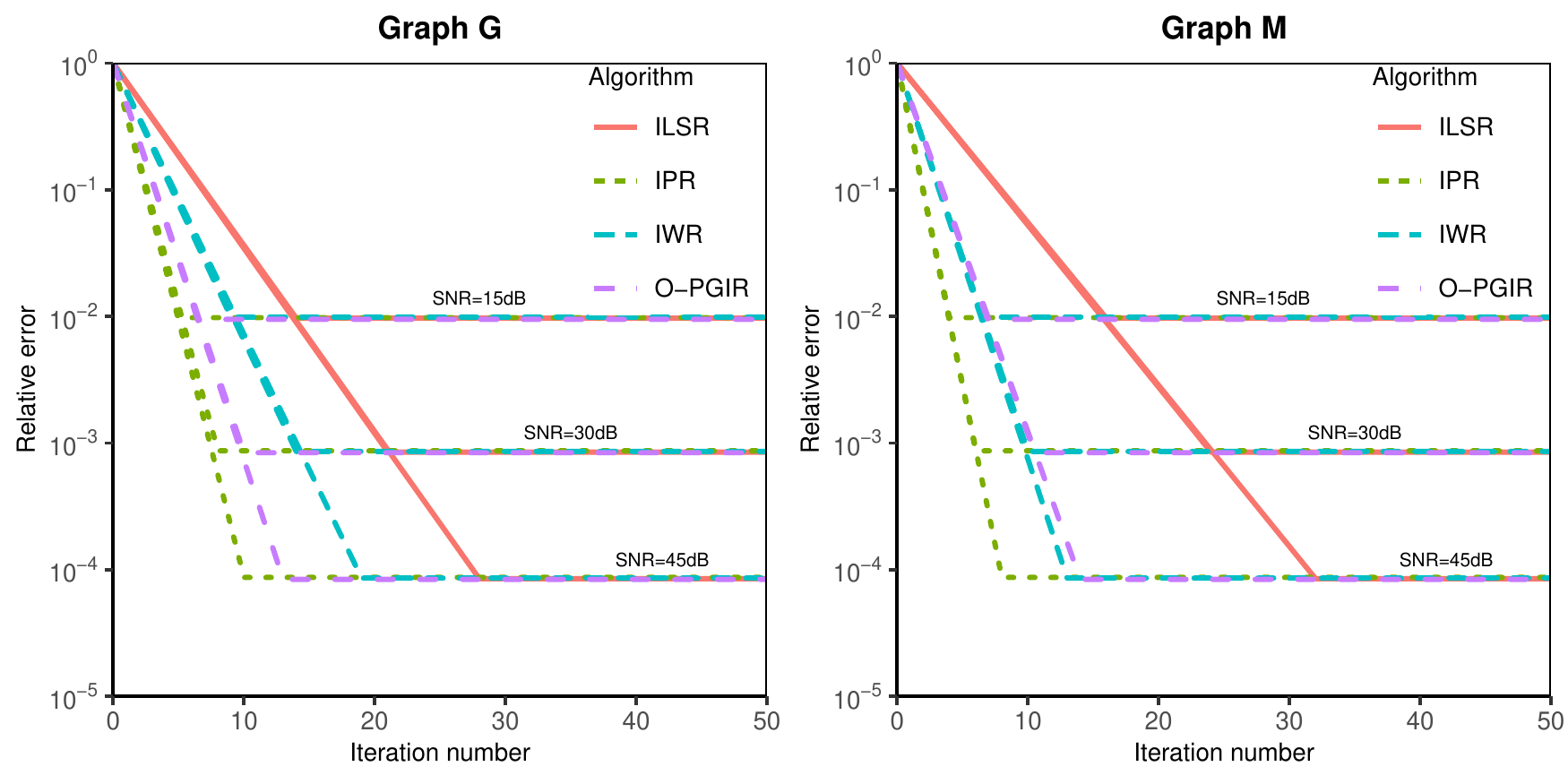}}
\caption{Left: graph $\mathcal{G}$. Right: graph $\mathcal{M}$. Robustness of the algorithms against noise with different SNR.}\label{fi:noise}
\end{figure}

\section{Conclusions}\label{se:con}
In this paper, the task of bandlimited signal reconstruction on graph from only a partial set of samples is investigated. An iterative method, called O-PGIR, is proposed to reconstruct the missing data from the observed samples. Similar to existing graph signal reconstruction algorithms, namely ILSR, IWR and IPR, the proposed method is derived from the Papoulis-Gerchberg iterative scheme, but with the optimal value of the relaxation parameter involved in the iteration step. Compared with the aforementioned algorithms, our method achieves similar or better performance both in terms of convergence rate and execution time.

\clearpage
\section*{References}
\def\refname{}


\begin{thebibliography}{34}

 \bibitem{sand} A. Sandryhaila, and J. M. F. Moura, ``Big data analysis with signal processing on graphs,'' {\em IEEE Sig. Proc. Mag.}, vol. 31, no. 5, pp. 80--90, Sep, 2014.
 \bibitem{sand2} A. Sandryhaila, and J. M. F. Moura, ``Discrete signal processing on graphs,'' {\em IEEE Trans. Sig. Proc.}, vol. 61, no. 7, pp. 1644--1656, Apr, 2013.
  \bibitem{CC} D. I. Shuman, S. K. Narang, P. Frossard, A. Ortega, and P. Vandergheynst, ``The emerging field of signal processing on graphs: Extending high-dimensional data analysis to networks and other irregular domains,'' {\em IEEE Trans. Sig. Proc. Mag.}, vol. 30, no. 3, pp. 83--98, May, 2013.
 \bibitem{pesen} I. Pesenson, ``Sampling in Paley-Wiener spaces on combinatorial graphs,'' {\em Trans. Am. Math. Soc.}, vol. 360, no. 10, pp. 5603--5627, Oct, 2008.
 \bibitem{nara} S. K. Narang, A. Gadde, E. Sanou, and A. Ortega, ``Signal processing techniques for interpolation of graph structured data,'' in {\em IEEE ICASSP}, Vancouver, BC, Canada, 2013, 5445--5449.
 \bibitem{anis} A. Anis, A. Gadde, and A. Ortega, ``Efficient sampling set selection for bandlimited graph signals using graph spectral proxies,'' {\em IEEE Trans. Sig. Proc.},  vol. 64, no. 14, pp. 3775--3789, Jul, 2015.
 \bibitem{chen} S. Chen, R. Varma, A. Sandryhaila, and J. Kova\v{c}evi\'{c}, ``Discrete signal processing on graphs: sampling theory,'' {\em IEEE Trans. Sig. Proc.}, vol. 63, no. 24, pp. 6510--6523, Dec, 2015.
 \bibitem{chen2} S. Chen, A. Sandryhaila, J. M. F. Moura, and J. Kova\v{c}evi\'{c}, ``Signal recovery on graphs: variation minimization,'' {\em IEEE Trans. Sig. Proc.}, vol. 63, no. 7, pp. 4609--4624, Sep, 2015.
 \bibitem{marq} A. G. Marques, S. Segarra, G. Leus, and A. Ribeiro, ``Sampling of graph signals with successive local aggregations,'' {\em IEEE Trans. Sig. Proc.}, vol. 64, no. 7, pp. 1832--1843, Apr, 2016.
 \bibitem{BB} S. K. Narang, A. Gadde, E. Sanou, and A. Ortega, ``Localized iterative methods for interpolation in graph structured data,'' in {\em IEEE GlobalSIP}, Austin, TX, USA, 2013, pp. 491--494.
  \bibitem{SMLR} S. Segarra, A. G. Marques, G. Leus, and A. Ribeiro, ``Reconstruction of Graph Signals Through Percolation from Seeding Nodes,'' {\em IEEE Trans. Sig. Proc.}, vol. 64, no. 16, pp. 4363--4378, Aug, 2016.
  \bibitem{TBDL}  M. Tsitsvero, S. Barbarossa, and P. Di Lorenzo, ``Signals on Graphs: Uncertainty Principle and Sampling,'' {\em IEEE Trans. Sig. Proc.}, vol. 64, no. 18, pp. 4845--4860, Sep, 2016.
 \bibitem{DD} X. Wang, P. Liu, and Y. Gu, ``Local-set-based graph signal reconstruction,'' {\em IEEE Trans. Sig. Proc.}, vol. 63, no. 9, pp. 2432--2444, May, 2015.
 \bibitem{wang2} X. Wang, J. Chen, and Y. Gu, ``Local measurement and reconstruction for noisy bandlimited graph signals,'' {\em Sig. Proc.}, vol. 129, pp. 119--129, Dec, 2016.
 \bibitem{zhu} X. Zhu, and M. Rabbat, ``Approximating signals supported on graphs,'' in {\em IEEE ICASSP}, Kyoto, Japan, 2012, 3921--3924.
 \bibitem{youl} D. C. Youla, and H. Webb, ``Image restoration by the method of convex projections: Part I - Theory,'' {\em IEEE Trans. Medical Imaging}, vol. 1, no. 2, pp. 81--94, Oct, 1982.
 \bibitem{chri} O. Christensen, ``Frames in Hilbert Spaces,'' in {\em An introduction to frames and Riesz bases}, Boston, MA, USA, Birkh\"{a}user, 2003, ch. 4, sec. 5.9, pp. 117--119.
 \bibitem{feic} H. G. Feichtinger, and K. H. Gr\"{o}chenig, ``Theory and practice of irregular sampling,'' {\em Math. Appl}, pp. 305--363, Jan, 1994.
 \bibitem{groc3} K. H. Gr\"{o}chenig, ``Reconstruction algorithms in irregular sampling,'' {\em Math. Comput.}, vol. 59, no. 199, pp. 181--194, Jul, 1992.
 \bibitem{gerc} R. W. Gerchberg, ``Super-resolution through error energy reduction,'' {\em Optica Acta}, vol. 21, no. 9, pp. 709--720, 1974.
 \bibitem{papou} A. Papoulis, ``A new algorithm in spectral analysis and band-limited extrapolation,'' {\em IEEE Trans. Circuits Syst.}, vol. 22, no. 9, pp. 735--742, Sep, 1975.
 \bibitem{AA} F. Marvasti, M. Analoui, and M. Gamshadzahi, ``Recovery of signals from nonuniform samples using iterative methods,'' {\em IEEE Trans. Sig. Proc.}, vol. 39, no. 4, pp. 872--878, Apr, 1991.
 \bibitem{zhou} D. Zhou, and B. Sch\"{o}lkopf, ``Learning from labeled and unlabeled data using random walks,'' in C.E. Rasmussen, H.H. Bülthoff, B. Schölkopf, M.A. Giese (eds) {\em Pattern Recognition DAGM}, LNCS, vol. 3175, 2004, pp.237--244.
 \bibitem{lim} T. Lim, ``Nonexpansive Matrices with Applications to Solutions of Linear Systems by Fixed Point Iterations'', {\em Fixed Point Theory Appl}, vol. 2010, no. 821928, 2010.
 \bibitem{varg} R. S. Varga, ``Basic Iterative Methods and Comparison Theorems,'' in {\em Matrix iterative analysis,} Springer, 2000, ch. 3, pp. 63--110.
 \bibitem{ferre} P. J. S. G. Ferreira, ``Interpolation and the discrete Papoulis-Gerchberg algorithm,'' {\em IEEE Trans. Sig. Proc.}, vol. 42, no. 10, pp. 2596--2606, Oct, 1994.
 \bibitem{marv2} P. J. S. G. Ferreira, ``Iterative and Noniterative Recovery of Missing Samples for 1-D Band-Limited Signals,'' in F. Marvasti (ed.) {\em Nonuniform sampling: theory and practice}, New York, NY, USA, Springer, ch. 4, pp. 235--281.
 \bibitem{erdos} P. Erd\"{o}s and P. R\'{e}nyi, ``On the evolution of random graphs,'' {\em Publ. Math. Inst. Hungary. Acad. Sci.}, vol. 5, pp. 17--61, 1960.
 \bibitem{glei} D. Gleich, The MatlabBGL Matlab Library [Online]. Available: \url{http://www.cs.purdue.edu/homes/dgleich/packages/matlab_bgl/index.html}
 \bibitem{cho} L. Cho, C. Chen, C. Lin and C. Hsu, ``The convergence analysis of extended Papoulis-Gerchberg Algorithm on AWGN-smeared signals,'' {\em 2018 IEEE International Conference on Applied System Invention (ICASI)}, Chiba, pp. 730--733, 2018.

\end{thebibliography}
\end{document}